\documentclass{article}

\usepackage{amsthm}
\usepackage{amsmath}
\usepackage{amssymb}
\usepackage{verbatim}

\newtheorem{theorem}{Theorem}
\newtheorem{lemma}[theorem]{Lemma}
\newtheorem{corollary}[theorem]{Corollary}

\newtheoremstyle{example}{\topsep}{\topsep}%
{}%         Body font
{}%         Indent amount (empty = no indent, \parindent = para indent)
{\bfseries}% Thm head font
{.}%        Punctuation after thm head
{2mm}%     Space after thm head (\newline = linebreak)
{\thmname{#1}\thmnumber{ #2}\thmnote{ #3}}%         Thm head spec

\theoremstyle{example}

\newcommand{\R}{\mathbb{R}}
\newcommand{\C}{\mathbb{C}}
\newcommand{\D}{\mathbb{D}}
\renewcommand{\Re}{\mathop\mathrm{Re}}
\renewcommand{\phi}{\varphi}

\title{Variation of Green Functions: Normal Derivatives}
\date{\today}

\author{Charles Z. Martin}

\begin{document}

\maketitle

\begin{abstract}
  We derive a variational formula for the outward normal derivative of the Green
  function for the Schr\"odinger and Laplace--Beltrami operators, viewed as perturbations of the Laplacian.
  As an application we begin to characterize elliptic growth---the growth of a domain
  pushed outward by its own Green function.
\end{abstract}

\maketitle
\section{Introduction}

The outward normal derivative of the Green function of a domain is of central importance to
  boundary--value problems associated to the underlying elliptic operator. Given a domain $D\subset \C$, a point $w\in D$,
  and an elliptic differential operator $L$, we define the Green function $g_w: \overline{D}-\{w\} \to \R$ as the solution (if it exists)
  to
  \begin{equation*}
    \begin{cases}
      Lg_w = \delta_w & \textrm{ in } D \\
      g_w = 0 & \textrm{ on }\partial D,
    \end{cases}
  \end{equation*}
  where $\delta_w$ denotes the Dirac delta supported at $w$. Occasionally we will denote
  $g_w(z)$ by $g(z,w)$, emphasizing the functional dependence upon both inputs. The existence of a Green function for a domain is tied to the regularity of the
  boundary. We will work only with smooth and analytic boundaries, which guarantees existence of Green functions and allows use of Green's
  theorem.
  
  The Green function is used to solve the inhomogeneous boundary--value problem for $L$ with zero boundary data. That is, a solution to
  \begin{equation*}
    \begin{cases}
      Lu = f & \textrm{ in } D \\
      u = 0 & \textrm{ on }\partial D,
    \end{cases}
  \end{equation*}
  is given by
  \begin{equation*}
    u(w) = \int_D g_w f\, dA.
  \end{equation*}
  Of equal importance is the homogeneous boundary--value problem
  \begin{equation} \label{dirichlet}
    \begin{cases}
      Lu = 0 & \textrm{ in } D \\
      u = f & \textrm{ on }\partial D.
    \end{cases}
  \end{equation}
  In some situations, the Green function can be used to extract a solution to this problem as well. In the case of the Laplacian, $L=\Delta$,
  a solution to \eqref{dirichlet} is given by
  \begin{equation*}
    u(w) = \int_{\partial D} f \partial_n g_w \, ds,
  \end{equation*}
  where $\partial_n$ denotes the outward normal derivative. The integral kernel $P_\zeta(w) = \partial_n g_w(\zeta)$ is called the Poisson
  kernel of the region $D$, and for simplicity of notation we sometimes use $P_\zeta(w)$ in place of $\partial_n g_w(\zeta)$, especially
  when we wish to view $\zeta$ as a fixed parameter.
  
  If instead we consider an operator in divergence form, $L=\nabla\lambda\nabla$,
  then a solution to \eqref{dirichlet} is given by
  \begin{equation*}
    u(w) = \int_{\partial D} f \lambda \partial_n g_w \, ds.
  \end{equation*}
  The normal derivative of the Green function also appears in the dynamics of moving boundaries. In \cite{Khavinson}
  the authors define a process wherein a domain grows with a boundary velocity determined by the outward derivative of its own Green
  function. This so--called elliptic growth is used to model real processes such as electrodeposition, crystal formation, and quasi--static
  fluid flowing through a porous medium. When the underlying operator is the Laplacian the growth is commonly called Laplacian growth
  or Hele--Shaw flow---see \cite{Mineev} for a survey. When the underlying operator is more general however,
  much less work has been done. It is known for example that lemniscates cannot survive a Laplacian growth process (see \cite{Khavinson2}),
  but a carefully--chosen elliptic growth might preserve them. To approach such problems we need an understanding of how the Green function
  depends upon the operator.
  
  Compared to that of a general elliptic operator,
  the Green function of the Laplacian is well--understood. In \cite{Martin} the author develops formulas for Green functions
  under various perturbations of the Laplacian; here we follow the same paradigm, studying the normal derivative of the Green function
  under a perturbation of the operator.

\section{Variation of the Normal Derivative}
\subsection{Schr\"odinger}

Our starting point is the variational formula derived in \cite{Martin}.
Consider a bounded domain $D$ with smooth, analytic boundary,
a positive function $u\in C^\infty(\overline{D})$
and $z,w\in D$. For $\epsilon > 0$ the Green function $g^*_w$ of $\Delta-\epsilon u$ satisfies
\begin{equation}
  g^*_w(z) = g_w(z) + \epsilon \int_D u(\xi)g(z,\xi)g(w,\xi)\, dA(\xi) + o(\epsilon) \label{Schrodinger}
\end{equation}
as $\epsilon\to 0$,
where $g_w$ denotes the Green function for the Laplacian and
the error term converges uniformly in $z$ for each fixed $w$. Furthermore, a full series expansion is given by
\begin{equation*}
  g^*_w = \sum_{n=0}^\infty (TM)^n g_w,
\end{equation*}
where we've used the operators $M: \phi\mapsto u\phi$ and
\begin{equation*}
  T:\phi \mapsto \int_D g_w\phi\, dA.
\end{equation*}
To derive a similar formula for $\partial_n g^*_w$ our major tool is the following lemma.

\begin{lemma}\label{NormalLemma}
  Let $D\subset \C$ be a bounded domain with smooth, analytic boundary and $f\in C^1(\overline{D})\cap C^2(D)$.
  Suppose further that $f=0$ on $\partial D$. Then for $\zeta\in \partial D$ we have
  \begin{equation*}
    \frac{\partial f}{\partial n}(\zeta) = \int_D \Delta f(z)P(z,\zeta)\, dA(z),
  \end{equation*}
  where $P(z,\zeta) = \partial_n g_z(\zeta)$ is the Poisson kernel of $D$.
\end{lemma}
\begin{proof}
  Let $A\subset \partial D$ be measurable with respect to arc length on $\partial D$.
  We define an auxiliary function $u$ which solves the Dirichlet problem
  \begin{equation*}
    \begin{cases}
      \Delta u = 0 & \textrm{ in } D \\
      u = \chi_A & \textrm{ on } \partial D
    \end{cases}
  \end{equation*}
  where $\chi_A$ denotes the characteristic function of $A$. From the Green identity
  \begin{equation}
    \int_{\partial D} \left( u\frac{\partial f}{\partial n} - f\frac{\partial u}{\partial n}\right)\, ds
      = \int_D \left(u\Delta f - f\Delta u\right)\, dA \label{BeltramiIdentity}
  \end{equation}
  we use the assumptions on $u$ and $f$ to find
  \begin{equation}
    \int_{A} \frac{\partial f}{\partial n}\, ds = \int_D u\Delta f\, dA. \label{PoissonLemma}
  \end{equation}
  Given $z\in D$ we return to the Green identity \eqref{BeltramiIdentity} with the functions $g_z$ and $u$. It follows that
  \begin{equation*}
    u(z) = \int_{\partial D} u\frac{\partial g_z}{\partial n}\, ds = \int_A \frac{\partial g_z}{\partial n}\, ds.
  \end{equation*}
  Inserting this into \eqref{PoissonLemma} and swapping the order of integration gives
  \begin{equation*}
    \int_{A} \frac{\partial f}{\partial n}\, ds = \int_A \left[\int_D \Delta f(z)  \frac{\partial g}{\partial n}(z,\zeta) \, dA(z)\right] \, ds(\zeta).
  \end{equation*}
  Since $A$ is arbitrary, we conclude
  \begin{equation*}
    \frac{\partial f}{\partial n}(\zeta) = \int_D \Delta f(z)\frac{\partial g}{\partial n}(z,\zeta)\, dA(z)
  \end{equation*}
  almost everywhere with respect to arc length.
  Our assumptions on $f$ imply that each side of this equation is a continuous function of $\zeta$.
  Hence the equation in fact holds everywhere.
\end{proof}

Before we can proceed with to the main result, we need to ascertain the regularity of the relevant Green functions. To use
lemma \ref{NormalLemma} we need to be sure the Green function for Schr\"odinger operators are continuously differentiable in some neighborhood
of the closure of the domain.

\begin{lemma}
  Suppose that $D\subset \C$ is a bounded domain with smooth, analytic boundary. Then the Green function of the Laplacian on $D$ is analytic in a neighborhood of $\overline{D}$ except at the
  singularity.
\end{lemma}
\begin{proof}
  Fix $w\in D$. Since $\partial D$ is smooth analytic, for each $\zeta\in\partial D$ there is a neighborhood $N_\zeta$ of $\zeta$
  and a conformal map $\phi:N_\zeta\to \D$ so that $\phi(N_\zeta\cap D) = \D^+$, the upper half--disk. Taking the union of all such
  $N_\zeta$ and $D$
  gives $\tilde{D}$, a neighborhood of $\overline{D}$. We can extend $g_w$ to $\tilde{D}$ as follows.
  
  Given $\zeta\in \partial D$ find the corresponding neighborhood $N_\zeta$ and
  conformal map $\phi$. Since $g_w\circ \phi^{-1}$ is harmonic on the upper--half disk, 0 on the interval $(-1,1)$,
  and continuous up to the real axis, the reflection principle for harmonic functions
  allows us to extend $g_w\circ \phi^{-1}$ to a harmonic function on $\D$ via the definition
  \begin{equation*}
    g_w\circ \phi^{-1} (z) = -g_w\circ \phi^{-1} (\overline{z})
  \end{equation*}
  whenever $z\in \D$ is below the real axis.
  Then $g_w = (g_w\circ \phi^{-1})\circ \phi$ extends to
  a harmonic function on $D \cup N_\zeta$. Now suppose that $z\in\C-\overline{D}$ is
  a point in both $N_\zeta$ and $N_\xi$; we need to show that the
  extensions arising from $\zeta$ and $\xi$ agree at $z$.
  Consider the conformal maps $\phi: N_\zeta\to \D$ and $\psi: N_\xi\to \D$. Then $\psi\circ\phi^{-1}$ is a map
  preserving the unit disk and the real line; hence for some real $a\in (-1,1)$ we can define $M_a:\D\to\D$ as
  \begin{equation*}
    M_a(t) = \frac{t-a}{1-at}
  \end{equation*}
  so that $\psi\circ \phi^{-1}=M_a$. We can rewrite the extension of $g_w$ resulting from $N_\xi$ as
  \begin{align*}
    (g_w\circ \psi^{-1}) (\overline{\psi(z)}) &= (g_w \circ \phi^{-1}\circ M_a^{-1})(\overline{M_a(\phi(z))}) \\
    &= (g_w \circ \phi^{-1}\circ M_a^{-1})(M_a(\overline{\phi(z)})) \\
    &= (g_w \circ \phi^{-1})(\overline{\phi(z)}),
  \end{align*}
  as desired. Since $g_w$ is harmonic away from the singularity, it is analytic there.
\end{proof}

\begin{lemma}
  Suppose that $D\subset \C$ is a bounded domain with smooth, analytic boundary.
  Let $u$ be a positive, smooth function in a neighborhood of $\overline{D}$.
  Then the Green function of $D$ for the operator $\Delta - u$ is
  smooth in a neighborhood of $\overline{D}$ except at the
  singularity. In particular, the Green function satisfies the conditions of lemma \ref{NormalLemma}.
\end{lemma}
\begin{proof}
  Fix a point $w\in D$. We can define
  \begin{equation*}
    T: \phi \mapsto \int_D g_z\phi\, dA
  \end{equation*}
  which as a left inverse to the Laplacian is a map from the Sobolev space $H^k(D)$ to $H^{k+2}(D)$ for any $k\geq 0$. Define
  $M$ to be the multiplication operator arising from $u$. Since $u$ is smooth in a neighborhood of $\overline{D}$,
  the operator $M$ maps $H^k(D)$ into itself. It was shown in \cite{Martin} that
  \begin{equation}
    g^*_w = \sum_{n=0}^\infty (TM)^n g_w \label{series}
  \end{equation}
  holds throughout $D$. Since $g_w$ is smooth in $D-\{w\}$, so is $g^*_w$. Furthermore, since $g_w$ extends smoothly
  into a neighborhood of $\overline{D}$ excluding $w$, equation \eqref{series} gives a smooth extension of $g^*_w$ into the same domain.
\end{proof}

\begin{theorem}\label{NormalTheorem}
  Let $D \subset \C$ be a bounded domain with smooth, analytic boundary and fix $w\in D$.
  Suppose that $u\in C^\infty(\overline{D})$ is a positive function.
  The outward normal derivative of the Green function $g^*$ of the Schr\"odinger operator $\Delta-\epsilon u$ satisfies
  \begin{equation*}
     \partial_n g^*_w(\zeta) = \partial_n g_w(\zeta) + \epsilon \int_D u(z)g(z,w)P(z,\zeta)\, dA(z) + o(\epsilon)
  \end{equation*}
  as $\epsilon \to 0$, where the convergence of $o(\epsilon)$ is uniform in $\zeta$ for each fixed $w$.
\end{theorem}
\begin{proof}
  Notice that $g^*_w$ satisfies
  \begin{equation*}
    \begin{cases}
      \Delta g_w^* = \delta_w + \epsilon ug_w^* & \textrm{ in } D \\
      g_w^* = 0 & \textrm{ on } \partial D
    \end{cases}
  \end{equation*}
  Using lemma \ref{NormalLemma} gives
  \begin{align}
    \partial_n g^*_w(\zeta) &= \int_D \left(\delta_w + \epsilon u(z)g_w^*(z)\right) P(z,\zeta)\, dA(z) \nonumber \\
      &= \partial_n g(\zeta,w) + \epsilon \int_D u(z)g_w^*(z) P(z,\zeta)\, dA(z), \label{SchrodingerNormalExact}
  \end{align}
  where we have used the fact that $P(z,\zeta) = \partial_n g_z(\zeta)$.
  Inserting the perturbation formula of equation \eqref{Schrodinger} into \eqref{SchrodingerNormalExact} gives
  \begin{equation*}
    \partial_n g_w^*(\zeta) = \partial_n g(\zeta,w) + \epsilon \int_D u(z)g_w(z) P(z,\zeta)\, dA + o(\epsilon). \qedhere
  \end{equation*}
\end{proof}
We remark that this formula can be easily derived from formal manipulation; beginning with equation \eqref{Schrodinger}
take the outward normal derivative in the $z$ variable at a point $\zeta$ on the boundary. Assuming $o(\epsilon)$ remains small and
the derivative commutes with the integral, we produce the result.
The above proof actually produces an expression for the second variation of $\partial_n g^*_w$, which we give here as a corollary.
\begin{corollary}
  Under the same conditions as the previous theorem, the second variation of $\partial_n g^*_w$ is given by
  \begin{equation*}
    \delta^2 \partial_n g^*_w(\zeta) = \int_D\int_D u(z) u(\xi) g(\xi,w) g(\xi,z) P(z,\zeta)\, dA(\xi) dA(z).
  \end{equation*}
\end{corollary}
As a further corollary, we can deduce a monotonicity result.
\begin{corollary}
  Under the same conditions as the previous theorem, the first variation $\delta\partial_n g^*_w(\zeta)$ is nonpositive at each
  point of $\partial D$.
\end{corollary}
\begin{proof}
  For any $z,w\in D$ and $\zeta\in \partial D$ we have $P(z,\zeta), u(z) \geq 0$ and $g(z,w)\leq 0$. Thus
  \begin{equation*}
    \int_D u(z)g(z,w) P(z,\zeta) \, dA(z) \leq 0.
  \end{equation*}
  The result follows from theorem \ref{NormalTheorem}.
\end{proof}

\subsection{Laplace--Beltrami}

Having completed the analysis of the Schr\"odinger case, we turn to a different perturbation of the Laplacian.
As before, consider a bounded domain $D\subset\C$ with smooth, analytic boundary and $u\in C^\infty(\overline{D})$. Given
$\epsilon>0$ define $\lambda=1+\epsilon u$ and the Laplace--Beltrami operator $L=\nabla\lambda\nabla$.
Let $g^*_w$ denote the Green function for $L$ with singularity at $w\in D$; we wish to derive a perturbation formula
for $\partial_n g^*_w(\zeta)$ when $\epsilon$ is close to 0. The result we obtained for the Schr\"odinger operator
turned out to be the obvious one that results from formal manipulation. In this vein, we can anticipate a Laplace--Beltrami analogue
by using the following perturbation formula from \cite{Martin}:
\begin{equation}
  g^*_w(z) = g_w(z) + \epsilon\int_D u\nabla g_z\cdot \nabla g_w\, dA + o(\epsilon)
\end{equation}
as $\epsilon \to 0$, where $w,z\in D$ and $o(\epsilon)$ converges uniformly in $z$ for each fixed $w$.
Let's take a normal derivative in $z$ at the point $\zeta\in \partial D$, assuming that $o(\epsilon)$ remains small and derivative commutes
with integral. We are lead to the formula
\begin{equation*}
   \partial_n g^*_w(\zeta) = \partial_n g_w(\zeta) + \epsilon\int_D u\nabla P_\zeta\cdot \nabla g_w\, dA + o(\epsilon).
\end{equation*}
Unfortunately this integral might diverge, but formal manipulation and integration by parts leads us to the following theorem.
\begin{theorem}
  Let $D \subset \C$ be a bounded domain with smooth, analytic boundary and fix $w\in D$.
  Suppose that $u\in C^\infty(\overline{D})$ is a positive function and define $\lambda(z) = 1+\epsilon u(z)$ for $\epsilon > 0$.
  The outward normal derivative of the Green function $g^*$ of the Laplace--Beltrami operator $L=\nabla\lambda\nabla$ satisfies
  \begin{equation*}
     \partial_n g^*_w(\zeta) =
     \partial_n g_w(\zeta) + \frac{\epsilon}{2}\left[\int_D \Delta u g_w P_\zeta\, dA - \partial_n g_w(\zeta)\left[u(\zeta)+u(w)\right]
      \right] + o(\epsilon)
  \end{equation*}
  as $\epsilon \to 0$, where the convergence of $o(\epsilon)$ is uniform in $\zeta$ for each fixed $w$.
\end{theorem}
\begin{proof}
  For a rigorous proof, we use a well--known change of variables to reduce the problem to that of the Schr\"odinger operator.
  Since $u\in C^\infty(\overline{D})\subset L^\infty(D)$, we consider only those $\epsilon$ satisfying
  $0\leq\epsilon \leq \epsilon_0 < \|u\|^{-1}_\infty$.
  This ensures that $\lambda \geq 0$ and that all of the following linear approximations have uniformly small error terms.
  If we define the auxiliary functions $G^*_w = g^*_w\sqrt{\lambda}$ and
  \begin{equation*}
    V = \frac{\Delta\sqrt{\lambda}}{\sqrt{\lambda}},
  \end{equation*}
  then we claim $\sqrt{\lambda(w)}G^*_w$ is the Green function for $\Delta - V$.
  Verification is a straight forward computation which we omit; a full derivation is given in \cite{Martin}.
  Furthermore, the dependence of $V$ upon $\epsilon$ can be seen to be $V = \epsilon\Delta u/2 + o(\epsilon)$.
  Using theorem \ref{NormalTheorem} we obtain
  \begin{equation}
    \partial_n \bigg[\sqrt{\lambda(z)\lambda(w)}g^*_w(z)\bigg]_{z=\zeta} =
      \partial_n g_w(\zeta) + \frac{\epsilon}{2}\int_D \Delta u g_w P_\zeta\, dA + o(\epsilon). \label{LB1}
  \end{equation}
  Using $g^*_w(\zeta) = 0$ the left--hand side of this equation simplifies to
  \begin{align*}
    \partial_n \bigg[\sqrt{\lambda(z)\lambda(w)}g^*_w(z)\bigg]_{z=\zeta} &=
      g^*_w(\zeta)\partial_n \bigg[\sqrt{\lambda(z)\lambda(w)}\bigg]_{z=\zeta} + \sqrt{\lambda(\zeta)\lambda(w)}\partial_n g_w^*(\zeta) \\
      &= \sqrt{\lambda(\zeta)\lambda(w)}\partial_n g_w^*(\zeta).
  \end{align*}
  We can rewrite \eqref{LB1} as
  \begin{equation}
    \partial_n g^*_w(\zeta)=
      (\lambda(\zeta)\lambda(w))^{-1/2}
      \left[\partial_n g_w(\zeta) + \frac{\epsilon}{2}\int_D \Delta u g_w P_\zeta\, dA + o(\epsilon)\right]. \label{LB2}
  \end{equation}
  We only want first--order dependence upon $\epsilon$, so we can expand
  \begin{equation*}
    (\lambda(\zeta)\lambda(w))^{-1/2} = 1 - \epsilon\left(\frac{u(\zeta)+u(w)}{2}\right) + o(\epsilon)
  \end{equation*}
  so that \eqref{LB2} becomes
  \begin{align*}
    \partial_n g^*_w(\zeta) &=
      \left[1 - \epsilon\left(\frac{u(\zeta)+u(w)}{2}\right) + o(\epsilon)\right]
      \left[\partial_n g_w(\zeta) + \frac{\epsilon}{2}\int_D \Delta u g_w P_\zeta\, dA + o(\epsilon)\right] \\
      &= \partial_n g_w(\zeta) + \frac{\epsilon}{2}\left[\int_D \Delta u g_w P_\zeta\, dA - \partial_n g_w(\zeta)\left[u(\zeta)+u(w)\right]
      \right] + o(\epsilon),
  \end{align*}
  as desired.
%%%
\end{proof}
As we did for the Schr\"odinger operator, we can deduce another monotonicity result.
\begin{corollary}
  Under the same conditions as the previous theorem, the variation $\delta\partial_n g^*_w(\zeta)$ satisfies
  \begin{equation*}
    \int_{\partial D} \delta\partial_n g^*_w \, ds \leq 0.
  \end{equation*}
\end{corollary}
\begin{proof}
  From the previous theorem,
  \begin{equation*}
    \int_{\partial D} \delta\partial_n g^*_w \, ds = \int_D \Delta u g_w\, dA - u(w) - \int_{\partial D} u\partial_n g_w\, ds,
  \end{equation*}
  where we have swapped integrals and used the fact that
  \begin{equation*}
    \int_{\partial D} \partial_n g_w \, ds = \int_{\partial D} P_\zeta(z) \, ds(\zeta) = 1
  \end{equation*}
  for each $w,z\in D$. The integral swap is valid since all integrands are of constant sign. Green's identity allows us to write
  \begin{equation*}
    \int_{\partial D} \delta\partial_n g^*_w \, ds = -2\int_{\partial D} u\partial_n g_w\, ds \leq 0,
  \end{equation*}
  since $u, \partial_n g_w \geq 0$ everywhere on $\partial D$.
\end{proof}

\section{Applications}
\subsection{The Dirichlet Problem}

For an elementary application of the Green variation, consider solving the Dirichlet problem on a domain $D$:
\begin{equation} \label{Dirichlet}
  \begin{cases}
    (\Delta-\epsilon u) \phi_\epsilon = 0 & \textrm{ in } D \\
    \phi_\epsilon = f &\textrm{ on } \partial D
  \end{cases}
\end{equation}
where $f$ is a bounded Borel function, $u\geq 0$ is smooth in a neighborhood of $\overline{D}$
and $\epsilon \geq 0$ is small. If we ignored the $\epsilon u$ term
in the problem---that is, assume $\epsilon$ is zero---we can estimate the error with the following theorem and corollary.
\begin{theorem}
  Let $D\subset \C$ be a bounded domain with smooth, analytic boundary and suppose $u\in C^\infty(\overline{D})$ is a positive function.
  For $\epsilon \geq 0$ and a bounded Borel function $f$ the solution $\phi_\epsilon$ to the Dirichlet problem \eqref{Dirichlet} satisfies
  \begin{equation*}
    \phi_\epsilon(z) = \phi_0(z) + \epsilon \int_D u \phi_0 g_z\, dA + o(\epsilon)
  \end{equation*}
  as $\epsilon \to 0$, where the error term converges uniformly in $z$.
\end{theorem}
\begin{proof}
  Let $g^*_w$ denote the Green function of the operator $\Delta-\epsilon u$ in the region $D$.
  Note that
  \begin{equation*}
    \begin{cases}
      (\Delta-\epsilon u)(\phi_\epsilon - \phi_0) = \epsilon u \phi_0 & \textrm{ in } D \\
      \phi_\epsilon - \phi_0 = 0 & \textrm{ on } \partial D,
    \end{cases}
  \end{equation*}
  whence an expression for $\phi_\epsilon$
  is given by
  \begin{equation*}
    \phi_\epsilon(z) = \phi_0(z) + \epsilon\int_D u\phi_0 g^*_z\, dA.
  \end{equation*}
  From the perturbation formula \eqref{series} we have
  \begin{equation*}
    \phi_\epsilon(z) = \phi_0(z) + \epsilon\int_D u\phi_0 g_z\, dA + o(\epsilon),
  \end{equation*}
  as desired.
\end{proof}
\begin{corollary}
  With the same assumptions as the previous theorem, the linearization of $\phi_\epsilon$ has the pointwise bound
  \begin{equation*}
    |\delta\phi_\epsilon(z)| \leq \|u\|_2\|g_z\|_2 \|f\|_\infty.
  \end{equation*}
\end{corollary}
\begin{proof}
  From the maximum modulus principle for harmonic functions, $\sup_D |\phi_0| \leq \sup_{\partial D} |f|$. The result follows from this and the
  Cauchy-Schwarz inequality.
\end{proof}

\subsection{Elliptic Growth and an Inverse Problem}

Consider a bounded domain $D\subset \C$ with smooth, analytic boundary and fix $w\in D$. Given a positive
function $u\in C^\infty(\overline{D})$, an elliptic growth process governed by the Schr\"odinger operator $\Delta - u$
is a nested sequence of domains
$\{D(t): t\geq 0\}$ such that $D(0) = D$ and
\begin{equation*}
  V_n(\zeta) = \partial_n g^*_w(\zeta),
\end{equation*}
where $V_n$ denotes the outward normal velocity of the boundary $\partial D(t)$ and $g^*_w$ denotes the Green function of the operator
$\Delta - u$ for the domain $D(t)$. That is, the Green function ``pushes'' the boundary of the domain outward, changing the domain---and hence
the Green function---in a nonlinear, nonlocal manner. In beginning to understand the nuances of this process, we ask the following question.
Which families of nested, growing domains can be produced by an elliptic growth process?

Our first approach to addressing this question is to consider the local problem. Given the initial domain $D$ we seek
which velocity fields $V$ can emerge from different choices of the function $u$. The intuition is as follows. Let $V_0$ be the velocity
field that arises from elliptic growth with $u\equiv 0$ (this is the well--known Laplacian growth process). If $V$ is another
velocity field close to $V_0$, can we choose $u\sim 0$ so that $V$ arises from elliptic growth via $\Delta - u$? In other words,
if we view the variation as a linear map $u\mapsto \delta\partial_n g^*_w$ on an appropriate function space, what is the image
of this map? We are led to the following theorem.
\begin{theorem}
  Let $D \subset \C$ be a bounded domain with smooth, analytic boundary and fix $w\in D$.
  % Suppose that $u\in C^\infty(\overline{D})$ is a positive function.
  Define
  \begin{equation*}
     Au(\zeta) = \int_D u(z)g(z,w)P(z,\zeta)\, dA(z).
  \end{equation*}
  Then $A$ is a linear map $L^2(D)\to L^2(\partial D)$ with dense range.
\end{theorem}
\begin{proof}
  To prove $A$ is defined on all of $L^2(D)$ it suffices to show that $g_wP_\zeta\in L^2(D)$ for each $w\in D$ and $\zeta\in \partial D$.
  Away from $w$ and $\zeta$ the function $g_wP_\zeta$ is smooth; furthermore, the singularity at $w$
  is logarithmic, which is square--integrable. It remains to understand the singularity at $\zeta$. Having assume the boundary of $D$ is
  smooth and analytic, there is a conformal map from a neighborhood of $\zeta$ to the unit disk which transforms the Poisson kernel
  and Green function canonically. Hence it suffices to understand the singularity at $\zeta$ in the case when $D=\D$, the unit disk.
  Without loss of generality, we can assume $w=0$ temporarily; a M\"obius transform can preserve the disk and send $w$ to 0,
  which doesn't disturb the behavior of the singularity at $\zeta$. Note that
  \begin{equation*}
    g_0(z)P(z,\zeta) = \frac{\ln|z|}{4\pi^2}\Re\left(\frac{\zeta+z}{\zeta-z}\right).
  \end{equation*}
  Taking the limit $z\to \zeta$ gives
  \begin{equation*}
    \lim_{z\to \zeta} g_0(z) P(z,\zeta) = -\frac{1}{2\pi^2},
  \end{equation*}
  so $g_0P_\zeta$ remains bounded near $\zeta$.
  At this point we return to the general domain $D$;
  we can write $g_w(z)P(z,\zeta) = h_\zeta(z)$, a bounded function in a neighborhood of $\overline{D}$
  except at $w$, where it has a logarithmic singularity.
  We conclude $g_wP_\zeta\in L^2(D)$.
  
  Next we show that $A$ maps into $L^2(\partial D)$. Note that
  \begin{equation*}
     \int_{\partial D} |Au|^2\, ds \leq \int_{\partial D} \|u\|_2 \|h_\zeta\|_2 \, ds(\zeta)
     \leq \|u\|_2 \sup_{\zeta\in \partial D} \|h_\zeta\|_2 \cdot \mathrm{length}(\partial D),
  \end{equation*}
  so it suffices to show that $\sup_{\partial D} \|h_\zeta\|_2 < \infty$. This follows easily from the fact that
  $\zeta \mapsto \|h_\zeta\|_2$ is a continuous function on the compact set $\partial D$.
  
  Next we show that the range of $A$ is dense in $L^2(\partial D)$. Suppose $v\in L^2(\partial D)$ is chosen so that
  \begin{equation*}
    \int_{\partial D} v(\zeta) (Au)(\zeta)\, ds(\zeta) = 0
  \end{equation*}
  for all $u\in L^2(D)$. Using the definition of $A$ and swapping integrals gives
  \begin{equation*}
    \int_D u(z) \left(\int_{\partial D} v(\zeta) h_\zeta(z)\, ds(\zeta)\right)\, dA(z) = 0.
  \end{equation*}
  Since $u$ is arbitrary, the inner integral must vanish for almost all $z\in D$.
  Note that $h_\zeta(z) = g_w(z)P(z,\zeta)$ and $g_w(z) < 0$ for all $z\in D$, so
  \begin{equation*}
    \int_{\partial D} v(\zeta) P(z,\zeta)\, ds(\zeta) = 0
  \end{equation*}
  for almost all $z\in D$. This integral is a harmonic function in $z$ which is zero almost everywhere in $D$, hence
  everywhere by continuity. The domain is regular enough so that, when $z$ approaches $\partial D$, we recover the function $v$.
  Hence $v=0$ in $L^2(\partial D)$. This shows the orthogonal complement of the range of $A$ is trivial, so the range of $A$
  is dense in $L^2(\partial D)$.
\end{proof}
For physical reasons, we might prefer function spaces different from $L^2$. One natural candidate for velocity fields is $L^1(\partial D)$,
since the $L^1$ distance between velocity functions measures the difference in areas due to the domain growing in different ways.
Furthermore we could consider only smooth perturbation functions $u$, in the spirit of the original perturbation result. To these ends we have
the following corollary.
\begin{corollary}
  With the same hypotheses as the previous theorem, $A$ maps $C^\infty(\overline{D})\to L^1(\partial D)$ and has dense range.
\end{corollary}
\begin{proof}
  Since $\partial D$ has finite measure, $L^2(\partial D)$ is dense in $L^1(\partial D)$. Furthermore $C^\infty(\overline{D})$ is dense
  in $L^2(D)$, and the result follows.
\end{proof}
The inverse problem is not fully solved with these results. First we should only consider nonnegative functions $u$ and characterize the
range of the map $A$. Secondly, the restriction of the domain of $A$ to smooth functions begs the following question. Can every
smooth function in $L^1(\partial D)$ be obtained from $A$ when restricted to $C^\infty(\overline{D})$? Finally, which of these results
can be carried over to elliptic growth governed by a Laplace--Beltrami operator? In that setting the relevant integrals are more delicate and
require further analysis. All of these problems provide directions for future work.

\bibliographystyle{amsplain}

		 \textsc{Department of Mathematics,
		 University of California, Santa Barbara,
		 CA 93106} \\
\textit{E-mail:} \texttt{cmart07@math.ucsb.edu}

\end{document}